\documentclass[reqno,oneside,11pt]{amsart}

\usepackage[T1]{fontenc}
\usepackage{times,mathptm}
\usepackage{amssymb,epsfig,verbatim,mathtools,float,hyperref,xcolor,mathrsfs,calligra,faktor,blkarray, booktabs}
\theoremstyle{plain}
\usepackage[all]{xy}

\newtheorem{thm}{Theorem}[section]

\newtheorem{cor}[thm]{Corollary}
\newtheorem{pro}[thm]{Proposition}
\newtheorem{lem}[thm]{Lemma}

\newtheorem{proposition-principale}[thm]{Proposition principale}
\newtheorem{theoalph}{Theorem}

\theoremstyle{definition}
\newtheorem{defi}[thm]{Definition}

\newtheorem{rem}[thm]{Remark}
\newtheorem{rems}[thm]{Remarks}

\newtheorem{probs}[thm]{Problems}
\newtheorem{prob}[thm]{Problem}
\newtheorem{question}[thm]{Question}

\addtocounter{section}{0}              
\numberwithin{equation}{section}       
\usepackage[left,pagewise]{lineno}

\newcommand*{\newfaktor}[2]{
  \raisebox{0.5\height}{\ensuremath{#1}}
  \mkern-5mu\diagup\mkern-4mu
  \raisebox{-0.5\height}{\ensuremath{#2}}
}

\begin{document}

\setlength{\baselineskip}{0.54cm}         
\title[Algebraic properties of the group of germs of 
 diffeomorphisms]{Algebraic properties of the group \\ of germs of 
 diffeomorphisms}
\date{}

\author{Dominique Cerveau}
\address{Univ. Rennes, CNRS, IRMAR-UMR 6625, F-35000 Rennes, France}
\email{dominique.cerveau@univ-rennes1.fr}

\author{Julie D\'eserti}
\address{Universit\'e d'Orl\'eans, Institut Denis Poisson, route de Chartres, $45067$ Orl\'eans Cedex $2$, France}
\email{deserti@math.cnrs.fr}

\subjclass[2020]{57S05, 20E99}

\begin{abstract} 
We establish some algebraic properties
of the group $\mathrm{Diff}(\mathbb{C}^n,0)$ of 
germs of analytic diffeomorphisms of 
$\mathbb{C}^n$, and its formal completion
$\widehat{\mathrm{Diff}}(\mathbb{C}^n,0)$.
For instance we describe the commutator 
of~$\mathrm{Diff}(\mathbb{C}^n,0)$, but also prove that
any finitely generated subgroup of $\mathrm{Diff}(\mathbb{C}^n,0)$ 
is residually finite; we thus obtain some constraints
of groups that embed into $\mathrm{Diff}(\mathbb{C}^n,0)$. 
We show that $\widehat{\mathrm{Diff}}(\mathbb{C}^n,0)$
is an Hopfian group, and that 
$\widehat{\mathrm{Diff}}(\mathbb{C}^n,0)$ and
$\mathrm{Diff}(\mathbb{C}^n,0)$ are
not co-Hopfian. We end by the description of 
the automorphism groups of $\widehat{\mathrm{Diff}}(\mathbb{C},0)$,
and $\mathrm{Diff}(\mathbb{C},0)$.
\end{abstract}

\maketitle


\section*{Introduction}

Let $M^n$ be a complex manifold of dimension $n$, and
let $\mathcal{F}$ be a codimension $p$ holomorphic
foliation~$M^n$ with singular locus $\mathrm{Sing}(\mathcal{F})$. 
Suppose that $N\subset M$ is a submanifold of 
dimension $n-p$ invariant by $\mathcal{F}$. Then 
there exists a natural morphism 
\[
\mathrm{Hol}\colon\pi_1(N\smallsetminus\mathrm{Sing}(\mathcal{F}),\ast)\to\mathrm{Diff}(\mathbb{C}^p,0)
\]
the so-called holonomy representation. 
Suppose that $\ker\mathrm{Hol}$ is non trivial, and 
let $\gamma$ be an element of 
$\pi_1(N\smallsetminus\mathrm{Sing}(\mathcal{F}),\ast)\smallsetminus\{\mathrm{id}\}$ such that $\mathrm{Hol}(\gamma)=\mathrm{id}$.
Then, $\gamma$ can be lifted in the leaves near $N$ to 
some loops which are homotopically non trivial in these
leaves. In that situation the holonomy representation
gives topological and dynamical informations on the 
foliation. Another interesting fact is the following. 
Suppose that $p=1$, and that the image of $\mathrm{Hol}$
is an abelian linearisable group; then $\mathcal{F}$
is defined in a neighborhood of
$N\smallsetminus\mathrm{Sing}(\mathcal{F})$
by a closed meromorphic $1$-form 
(\cite{CerveauMattei, Pereira}).

As a consequence,
the study of such representations is an important 
problem, and requires the knowledge of algebraic 
properties of the groups $\mathrm{Diff}(\mathbb{C}^p,0)$.
In \cite{Dieudonne} Dieudonn\'e describes the geometric 
and algebraic properties of the classical linear
groups $\mathrm{GL}(n,\Bbbk)$. This text is part of a similar 
perspective by highlighting some properties that we considered important.

\bigskip

\textbf{Notations}. If $n$ is an integer, then 
$\mathrm{G}_0^n$ denotes $\mathrm{Diff}(\mathbb{C}^n,0)$, and 
$\widehat{\mathrm{G}}_0^n$ denotes its formal completion. Furthermore,~$\mathrm{G}_k^n$ is the set of elements of 
$\mathrm{G}_0^n$ tangent to the identity at order $k$,
and $\widehat{\mathrm{G}}_k^n$ is its formal completion. 

\bigskip

\textbf{Structure of the paper}. In 
\S\ref{sec:PoincareSiegel} we establish some consequences of 
Poincar\'e linearisation theorem:

\begin{theoalph}
{\sl For any $k\geq 1$ we have
\begin{align*}
& [\mathrm{G}_0^n,\mathrm{G}_k^n]=\mathrm{G}_k^n, && \text{and} && [\widehat{\mathrm{G}}_0^n,\widehat{\mathrm{G}}_1^n]=\widehat{\mathrm{G}}_1^n. 
\end{align*}

As a consequence, we obtain:
\[
[\mathrm{G}_0^n,\mathrm{G}_0^n]=\big\{f\in\mathrm{G}_0^n\,\vert\,\det Df_{(0)}=1\big\}.
\]}
\end{theoalph}

Finding the finitely generated groups that embed into 
$\mathrm{G}_0^n$ is a problem related to the foliations 
theory; in \S\ref{sec:residuallyfinite} we deal with this 
question and we get:

\begin{theoalph}
{\sl Any finitely generated subgroup of $\mathrm{G}_0^n$ $($resp.
$\widehat{\mathrm{G}}_0^n)$ is residually finite.

Hence, if $\mathrm{H}$ is a finitely generated and 
non residually finite group, then $\mathrm{H}$ does 
not embed into $\mathrm{G}_0^n$.}
\end{theoalph}

But any finitely generated residually finite group is 
a Hopfian group (\cite{Malcev1, Malcev2}); as a 
result, any finitely generated subgroup of $\mathrm{G}_0^n$
$($resp. $\widehat{\mathrm{G}}_0^n)$ is a Hopfian group.
In \S \ref{sec:hopcohop} we refine this result, and also 
look at the co-Hopfian property:

\begin{theoalph}
{\sl The group $\widehat{\mathrm{G}}_0^1$ is a Hopfian 
group.

The groups $\mathrm{G}_1^n$ and $\widehat{\mathrm{G}}_1^n$
are not co-Hopfian groups.}
\end{theoalph}

Inspired by \cite{Deserti2} we study in \S\ref{sec:autG01} 
the automorphism groups of $\mathrm{G}_0^1$ and
$\widehat{\mathrm{G}}_0^1$:

\begin{theoalph}
{\sl The group $\mathrm{Aut}(\widehat{\mathrm{G}}_0^1)$ is 
generated by the inner automorphisms and the automorphisms of 
the field $\mathbb{C}$. In other words 
\[
\mathrm{Out}(\widehat{\mathrm{G}}_0^1)\simeq\mathrm{Aut}(\mathbb{C},+,\cdot)
\]
where $\mathrm{Out}(\widehat{\mathrm{G}}_0^1)$ denotes
the non-inner automorphisms of $\widehat{\mathrm{G}}_0^1$.

The group $\mathrm{Out}(\mathrm{G}_0^1)$ is isomorphic
to $\newfaktor{\mathbb{Z}}{2\mathbb{Z}}$.}
\end{theoalph}

\section{Notations and Definitions}

Let $n$ be a positive integer; consider 
\[
\mathrm{G}_0^n=\mathrm{Diff}(\mathbb{C}^n,0)=\big\{f\colon\mathbb{C}^n_{,0}\to\mathbb{C}^n_{,0}\text{ holomorphic }\,\vert\, Df_{(0)}\in\mathrm{GL}(\mathbb{C}^n)\big\}.
\]
Denote by $\widehat{\mathrm{G}}_0^n$ the formal completion
of $\mathrm{G}_0^n$; in other words, $\widehat{\mathrm{G}}_0^n$
is the set of formal diffeomorphisms
$\widehat{f}=(\widehat{f}_1,\widehat{f}_2,\ldots,\widehat{f}_n)$ 
where 
\begin{itemize}
\item[$\diamond$] $\widehat{f}_i$ is a formal series, 
\item[$\diamond$] $\widehat{f}(0)=0$,
\item[$\diamond$] and $D\widehat{f}_{(0)}$ belongs to~$\mathrm{GL}(\mathbb{C}^n)$.
\end{itemize}
More generally, let $k$ be a positive integer; consider $\mathrm{G}_k^n$ the set of 
elements of $\mathrm{G}_0^n$ tangent
to the identity at order $k$ ({\it i.e.} 
$f$ belongs to $\mathrm{G}_k^n$ if and only if $f=\mathrm{id}+$ terms 
of order $\geq k+1$), and $\widehat{\mathrm{G}}_k^n$ be the
formal completion of $\mathrm{G}_k^n$.

Each $\mathrm{G}_k^n$'s (resp. $\widehat{\mathrm{G}}_k^n$'s)
is a normal subgroup of $\mathrm{G}_0^n$ 
(resp. $\widehat{\mathrm{G}}_0^n$). The quotients 
$\newfaktor{\mathrm{G}_0^n}{\mathrm{G}_k^n}$ and~$\newfaktor{\widehat{\mathrm{G}}_0^n}{\widehat{\mathrm{G}}_k^n}$ 
are isomorphic, and can be identified with the 
group of polynomial maps
\[
\mathrm{Pol}_k^n=\big\{f\colon\mathbb{C}^n\to\mathbb{C}^n\text{ polynomial }\,\vert\,f(0)=0,\, \deg f\leq k, \,Df_{(0)}\in\mathrm{GL}(\mathbb{C}^n)\big\}
\]
whose the law group is the law of composition truncated to 
order $k+1$.

Remark that $\newfaktor{\mathrm{G}_0^n}{\mathrm{G}_k^n}$ acts
faithfully on the space $\mathbb{C}[x_1,x_2,\ldots,x_n]_k$
of polynomials of degree less or equal than $k$
(still by truncated composition). Therefore, the 
$\newfaktor{\mathrm{G}_0^n}{\mathrm{G}_k^n}$'s can be identified
with subgroups of $\mathrm{GL}(\mathbb{C}[x_1,x_2,\ldots,x_n]_k)$.

For elements $g$ and $h$ of a group $\mathrm{H}$, the \textbf{commutator} of $g$ and $h$ is $[g,h]=ghg^{-1}h^{-1}$.
The \textbf{derived subgroup} $[\mathrm{H},\mathrm{H}]$ (also called the \textbf{commutator subgroup}) of $\mathrm{H}$ is the subgroup generated by all the commutators.
This construction can be iterated:
\[
\left\{
\begin{array}{ll}
\mathrm{H}^{(0)}:=\mathrm{H}\\
\mathrm{H}^{(n)}:=[\mathrm{H}^{(n-1)},\mathrm{H}^{(n-1)}]\quad n\in \mathbb{N} \end{array}
\right.
\]
The groups $\mathrm{H}^{(2)}$, $\mathrm{H}^{(3)}$, $\ldots$ are called the second derived subgroup, third derived subgroup, and so forth, and the descending normal series
\[
\cdots \triangleleft \mathrm{H}^{(2)}\triangleleft \mathrm{H}^{(1)}\triangleleft \mathrm{H}^{(0)}=\mathrm{H}
\]
is called the \textbf{derived series}. A \textbf{solvable group} is a group whose derived series terminates in the trivial subgroup.

The groups $\mathrm{G}_0^1$ and $\widehat{\mathrm{G}}_0^1$
contain free subgroups (\cite{BerthierCerveauLinsNeto}).
It's a bit surprising since $\widehat{\mathrm{G}}_0^1$ appears
as the limit of solvable groups
$\newfaktor{\mathrm{G}_0^1}{\mathrm{G}_k^1}$. To construct 
free subgroups the authors use the following deep 
result (\cite{Cohen}): consider the two following
homeomorphisms of $\mathbb{R}$
\begin{align*}
    & f_1(x)=x+1, && f_2(x)=x^3,
\end{align*}
then the group generated by $f_1$ and $f_2$ is a free group.

Note that $f_2$ is invariant by the involution $\frac{1}{x}$; 
hence the group generated by $\widetilde{f_1}(x)=\frac{x}{1+x}$
and $f_2(x)=x^3$ is free. As a consequence, the group generated
by the germs of analytic diffeomorphisms at the origin of~$\mathbb{R}$
\begin{align*}
    & \widetilde{f_1}(x)=\frac{x}{1+x}&& \text{and} && \widetilde{f_2}(x)=f_2^{-1}\widetilde{f_1}f_2(x)=\frac{x}{(1+x^3)^{1/3}} 
\end{align*}
is a free subgroup; it induces a free subgroup of $\mathrm{G}_0^1$ (and $\widehat{\mathrm{G}}_0^1$).

Denote by $p_k\colon\mathrm{G}_0^n\to\newfaktor{\mathrm{G}_0^n}{\mathrm{G}_k^n}$
the projection.

\smallskip

Let $\mathfrak{M}_n$ be the maximal ideal of 
$\mathcal{O}(\mathbb{C}^n,0)$ given by 
\[
\mathfrak{M}_n=\big\{f\in\mathcal{O}(\mathbb{C}^n,0)\,\vert\, f(0)=0\big\}.
\]
Denote by $\chi_0^n$ the set of germs of holomorphic vector fields at
the origin of $\mathbb{C}^n$, and by $\widehat{\chi}_0^n$ its
formal completion. Let $\chi_k^n$ be the subspace of $\chi_0^n$ 
made up of vector fields that vanish at $0$ at order $k-1$, 
that is $\chi_k^n=\mathfrak{M}_n^k\chi_0^n$. Finally, let us 
denote by $\widehat{\chi}_k^n$ the formal completion of 
$\chi_k^n$.

We recall some classical facts; proofs can be found for instance 
in \cite{CanoCerveauDeserti}. Let $X$ be an element 
of~$\widehat{\chi}_2^n$; denote by $\exp tX\in\widehat{\mathrm{G}}_1^n$
the one-parameter subgroup associated to $X$. By definition 
$\exp tX$ is solution of the O.D.E.
\[
\frac{\partial}{\partial t}\varphi_t(x)=X\big(\varphi_t(x)\big)
\]
with initial condition $\varphi_0(x)=x$.
It is easy too see that $\exp tX$ is "polynomial in the parameter
$t$", that is:
\[
\exp tX=\mathrm{Id}+\displaystyle\sum_{i=1}^\infty t^iA_i
\]
where the $A_i$'s belong to $\big(\mathfrak{M}_n^{i+1}\big)^n$.
In particular, $\exp X$ is well defined in $\widehat{\mathrm{G}}_1^n$.

\begin{pro}\label{pro:tec1}
{\sl The following properties hold.
\begin{itemize}
\item[$\diamond$] The map 
$\exp\colon\widehat{\chi}_2^n \to\widehat{\mathrm{G}}_1^n$ is bijective.

\item[$\diamond$] Let $X$, $Y$ be two elements of 
$\widehat{\chi}_2^n$. Then, $f=\exp X$ and $g=\exp Y$ commute if and only if $X$
and $Y$ commute.
\end{itemize}}
\end{pro}

\begin{rem}
The first property is no longer true in the holomorphic case: if $h$ belongs to 
$\mathrm{G}_1^1$ and $h=\exp X$, then $X$ is most of the time divergent 
(\cite{Ecalle}).
\end{rem}

In dimension $1$, there are normal forms for elements of 
$\chi_2^1$ and $\widehat{\chi}_2^1$ as follows:

\begin{pro}\label{pro:tec2}
{\sl Let $X$ be an element of $\chi_2^1$ $($resp. $\widehat{\chi}_2^1)$. Then, 
$X$ is holomorphically $($resp. formally$)$ conjugated to a vector field of the
type $\frac{x^{k+1}}{1+\lambda x^k}\frac{\partial}{\partial x}$, 
$k\in\mathbb{N}$, $k\geq 1$, $\lambda\in\mathbb{C}$.}
\end{pro}

The proof is a direct application of the implicit function 
theorem (\cite{CanoCerveauDeserti}). For instance the one-parameter subgroup of $x^2\frac{\partial}{\partial x}$ is
$\exp tx^2\frac{\partial}{\partial x}=\frac{x}{1-tx}$.

\bigskip

\section{Poincar\'e linearisation theorems and their consequences}\label{sec:PoincareSiegel}

Let $\lambda=(\lambda_1,\lambda_2,\ldots,\lambda_n)$ be a 
$n$-tuple of non-zero complex numbers. We say that $\lambda$ is
\textbf{without resonance} if the equality 
$\lambda_1^{p_1}\lambda_2^{p_2}\ldots\lambda_n^{p_n}=\lambda_j$, 
$p_\ell\in\mathbb{N}$, implies $p_j=1$, and $p_\ell=0$ for 
$\ell\not=j$. If $A\in\mathrm{GL}(n,\mathbb{C})$, we denote by 
$\mathrm{Spec}(A)$ the non-ordored set of its eigenvalues 
$(\lambda_1,\lambda_2,\ldots,\lambda_n)$.

\begin{thm}[Formal Poincar\'e Theorem, \cite{Arnold}]
{\sl Let $f$ be an element of $\widehat{\mathrm{G}}_0^n$. 
Set $A=Df_{(0)}$. 

Assume that $\mathrm{Spec}(A)$
is without resonance.

Then, $f$ is formally conjugate to $A$, {\it i.e.} there 
exists $\varphi\in\widehat{\mathrm{G}}_1^n$ such that 
$f=\varphi A\varphi^{-1}$.}
\end{thm}

\begin{thm}[Holomorphic Poincar\'e Theorem, \cite{Arnold}]
{\sl Let $\phi$ be an element of $\mathrm{G}_0^n$. 
Set $A=D\phi_{(0)}$.

Assume that $\mathrm{Spec}(A)$ is without resonance, and that 
$\mathrm{Spec}(A)$ is contained
either in the open unit disk $\mathbb{D}(0,1)$, or in its 
complement $\mathbb{C}\smallsetminus\mathbb{D}(0,1)$. 

Then, $\phi$ is holomorphically conjugate to $A$, {\it i.e.} there 
exists $\varphi\in\mathrm{G}_1^n$ such that 
$\phi=\varphi A\varphi^{-1}$.}
\end{thm}

We say that $\phi$ is \textbf{formally linearisable} 
(resp. \textbf{holomorphically linearisable}) in the 
formal (resp. holomorphic) case. In both cases the 
diffeomorphism $\varphi$ is called the 
\textbf{linearising map}.

\begin{rem}
There are generalisations of the previous result; the first 
are due to Siegel: when 
$\mathrm{Spec}(A)=(\lambda_1,\lambda_2,\ldots,\lambda_n)$ 
is contained neither in $\mathbb{D}(0,1)$,
nor in $\complement\mathbb{D}(0,1)$; it requires 
diophantine conditions controlling
$\big\vert\lambda_1^{p_1}\lambda_2^{p_2}\ldots\lambda_n^{p_n}-\lambda_j\big\vert$ 
that produce the convergence of the linearising maps
(\cite{Siegel1, Siegel2}). 
\end{rem}

\begin{rem}
In the Poincar\'e Theorems assume that $\phi$ can be 
written as $Ah$ where $h=\mathrm{id}+\ldots\in\mathrm{G}_k^n$;
then the linearising map $\varphi$ can be chosen in 
$\mathrm{G}_k^n$ (resp. $\widehat{\mathrm{G}}_k^n$ in the 
formal case). 
\end{rem}

Curiously while the proof of Poincar\'e Theorem comes from
analysis (in its holomorphic version) we deduce from it 
algebraic properties. 

\begin{thm}\label{thm:com}
{\sl We have $[\mathrm{G}_0^n,\mathrm{G}_1^n]=\mathrm{G}_1^n$, 
and similarly 
$[\widehat{\mathrm{G}}_0^n,\widehat{\mathrm{G}}_1^n]=\widehat{\mathrm{G}}_1^n$.}
\end{thm}

\begin{rem}
Theorem \ref{thm:com} has been proved in \cite{CanoCerveauDeserti} for $n=1$.
\end{rem}

\begin{proof}[{\sl Proof}]
Let us first remark that if $f$, $g$ are elements of 
$\mathrm{G}_0^n$ and $\mathrm{G}_1^n$ respectively, then
the commutator $[f,g]=fgf^{-1}g^{-1}$ is an element of 
$\mathrm{G}_1^n$.

Let $A$ be an element of $\mathrm{GL}(n,\mathbb{C})$ that 
satisfies the assumptions of Poincar\'e Theorem, for instance
$A=\lambda\,\mathrm{id}$ with $0<\vert\lambda\vert< 1$. If $h$
belongs to $\mathrm{G}_1^n$, then $hA$ is linearisable, 
that is there exists $\varphi\in\mathrm{G}_1^n$ such that 
\[
\varphi A\varphi^{-1}=hA.
\]
As a result, $[\varphi, A]=h$, and we get 
the result.

A similar proof works in the formal case.
\end{proof}

\begin{cor}
{\sl The commutator of $\mathrm{G}_0^n$ is given by:
\[
[\mathrm{G}_0^n,\mathrm{G}_0^n]=\big\{f\in\mathrm{G}_0^n\,\vert\, \det Df_{(0)}=1\big\}.
\]}
\end{cor}

\begin{proof}[{\sl Proof}]
The inclusion 
\[
\big\{f\in\mathrm{G}_0^n\,\vert\, \det Df_{(0)}=1\big\}\subset[\mathrm{G}_0^n,\mathrm{G}_0^n].
\]
holds. So, we just need to prove that any element $f$
such that $\det Df_{(0)}=1$ is a product of 
commutators of~$\mathrm{G}_0^n.$ Write $f=Ah$
with $A=Df_{(0)}\in\mathrm{SL}(n,\mathbb{C})$
and $h\in\mathrm{G}_1^n$. On the one hand, 
$A\in\mathrm{SL}(n,\mathbb{C})$ is a product of 
commutators of $\mathrm{GL}(n,\mathbb{C})\subset\mathrm{G}_0^n$; 
and, one the other hand, $h$ is a product
of commutators of $\mathrm{G}_0^n$ (Theorem~\ref{thm:com}).
Consequently, $f$ is a product of 
commutators of~$\mathrm{G}_0^n.$
\end{proof}

Theorem \ref{thm:com} can be generalised as follows:

\begin{thm}
{\sl Any element of $\mathrm{G}_k^n$ is the commutator
of an element of $\mathrm{G}_0^n$ and an element 
of~$\mathrm{G}_k^n$, {\it i.e.} $[\mathrm{G}_0^n,\mathrm{G}_k^n]=\mathrm{G}_k^n$.}
\end{thm}

We deduce from it the following statement that can be 
useful: 

\begin{pro}
{\sl Let $\mathrm{H}$ be a group, and let
$\tau\colon\mathrm{G}_0^n\to\mathrm{H}$ be a group 
homomorphism. 

Assume that there exists $f$ in $\ker\tau$ such that
$Df_{(0)}=A$ satisfies the assumptions of Poincar\'e
Theorem. Then $\mathrm{G}_1^n$ is 
contained in $\ker\tau$.}
\end{pro}

\begin{proof}[{\sl Proof}]
Let $h$ be an element of $\mathrm{G}_1^n$. Since $A$ satisfies the 
assumptions of Poincar\'e Theorem, then 
$hf$ and $f$ are conjugate: there exists
$\varphi$ in $\mathrm{G}_1^n$ such that 
$hf=\varphi f\varphi^{-1}$. As a consequence, 
$h=[\varphi,f]$ and $h$
belongs to~$\ker\tau$.
\end{proof}

\begin{rem}
Since $A$ and $f\in\ker\tau$ are conjugate, $A$ belongs 
to $\ker\tau$, and so does the normal subgroup of 
$\mathrm{GL}(n,\mathbb{C})$ generated by $A$. The quotient
$\newfaktor{\mathrm{G}_0^n}{\ker \tau}\simeq\mathrm{im}\,\tau$
can thus be identified with the quotient 
$\newfaktor{\mathrm{GL}(\mathbb{C}^n)}{\mathrm{GL}(\mathbb{C}^n)\cap\ker\tau}$.
\end{rem}

Poincar\'e theorem gives the description of 
germs of diffeomorphisms with generic linear part. 
When the linear part is not generic, for instance
in dimension $1$, the formal classification is relatively easy whereas the holomorphic one is rather 
difficult. There are many contributions (\cite{Ecalle, PerezMarco, Yoccoz, Malgrangebourbaki, ArnoldIlyashenko, Arnold}).

\bigskip

\section{Finite groups}

Finite subgroups of $\mathrm{G}_0^n$ are described by 
the following classical result:

\begin{thm}\label{thm:clas}
Let $\mathrm{H}$ be a finite subgroup of $\mathrm{G}_0^n$.
Denote by 
\[
\mathrm{H}_0=\big\{Df_{(0)}\,\vert\, f\in\mathrm{H}\big\}\subset\mathrm{GL}(n,\mathbb{C})
\]
the linear part of $\mathrm{H}$.
Then, $\mathrm{H}$ is holomorphically conjugate to 
$\mathrm{H}_0$, {\it i.e.} there exists
$\varphi\in\mathrm{G}_0^n$
such that $\mathrm{H}=\varphi^{-1}\mathrm{H}_0\varphi$.
\end{thm}

\begin{proof}
The proof relies on a classical average argument (like 
in the Cartan-Bochner's theorem about the local linearisation of the 
action of a compact Lie group near a fixed point).
Consider the map
$\varphi=\displaystyle\sum_{h\in\mathrm{H}}(Dh_{(0)})^{-1}\circ h$;
it is holomorphic, and 
$D\varphi_{(0)}=(\#\mathrm{H})\mathrm{id}$, so $\varphi$ 
belongs to $\mathrm{G}_0^n$.
Finally, let us remark that for any $f\in\mathrm{H}$ one 
has
\[
\varphi\circ f=\displaystyle\sum_{h\in\mathrm{H}}(Dh_{(0)})^{-1}\circ h\circ f=Df_{(0)} \displaystyle\sum_{h\circ f\in\mathrm{H}}(D(h\circ f)_{(0)})^{-1}\circ (h\circ f)=Df_{(0)}\circ\varphi.
\]
\end{proof}

\bigskip

\section{Residually finite groups}\label{sec:residuallyfinite}

There are a number of equivalent definitions of residually finite
groups; we will use the following one:

\begin{defi}
A group $\mathrm{H}$ is \textbf{residually finite} if for every
element $h$ in $\mathrm{H}\smallsetminus\{\mathrm{id}_{\mathrm{H}}\}$ 
there exists 
a group morphism $\varphi\colon\mathrm{H}\to\mathrm{F}$ 
from $\mathrm{H}$ to a finite 
group $\mathrm{F}$ such that $\varphi(h)\not=~\mathrm{id}_{\mathrm{F}}$.
\end{defi}

The 
groups $\mathbb{Z}$ and $\mathrm{SL}(n,\mathbb{Z})$ are
residually finite (by reduction modulo $p$). Subgroups of 
residually finite groups are residually finite.
Conversely, a non-finite simple group is not residually finite, 
and the Baumslag-Solitar group
\[
\mathrm{BS}(2,3)=\langle a,\, b\,\vert\, a^{-1}b^2a=b^3\rangle
\]
is not residually finite (\cite{BaumslagSolitar}).

A group is \textbf{linear} if it is isomorphic to a 
subgroup of $\mathrm{GL}(n,\Bbbk)$ where $\Bbbk$ is 
a field. 

Malcev established the following fundamental result:

\begin{thm}[\cite{Malcev1, Malcev2}]\label{thm:Malcev}
{\sl A finitely generated linear group is residually finite.}
\end{thm}

In particular, there is no faithful linear representation
of a non-residually finite, finitely generated group.

\begin{rem}
The assumption "finitely generated" turns out to be essential.
For instance, $\mathbb{Q}$ is not resi\-dually finite. If $\xi$
is a positive transcendental number, then 
$\xi^{\mathbb{Q}}=\big\{\xi^{\frac{p}{q}}\,\vert\,\frac{p}{q}\in\mathbb{Q}\big\}$ 
is isomorphic to $\mathbb{Q}$. As a subgroup of a residually
finite group is residually finite, we thus get that
$\mathrm{GL}(1,\mathbb{C})\simeq\mathbb{C}^*$, 
$\mathrm{GL}(n,\mathbb{C})$, $\mathrm{G}_0^n$
and $\widehat{\mathrm{G}}_0^n$ are not 
residually finite.
\end{rem}

The following statement, 
that is mentioned in \cite{ClaudonLorayPereiraTouzet} without 
detail in the $1$-dimensional and formal case,  is a direct consequence of Theorem 
\ref{thm:Malcev}:

\begin{thm}\label{thm:fgG0n}
{\sl Any finitely generated subgroup 
of $\mathrm{G}_0^n$ $($resp. $\widehat{\mathrm{G}}_0^n)$ 
is resi\-dually finite.}
\end{thm}

\begin{proof}[{\sl Proof}]
Let $\mathrm{H}$ be a finitely generated subgroup of $\mathrm{G}_0^n$,
and $h$ be an element of $\mathrm{G}_0^n\smallsetminus\{\mathrm{id}\}$.
There exists an integer $k$ such that $p_k(h)$ is non-trivial in
the quotient group $\newfaktor{\mathrm{G}_0^n}{\mathrm{G}_k^n}$. 
Recall that $\newfaktor{\mathrm{G}_0^n}{\mathrm{G}_k^n}$ is isomorphic
to a subgroup of a linear group; Theorem \ref{thm:Malcev} applied
to the group $p_k(\mathrm{H})$ asserts the existence of a 
morphism $\varphi_k\colon p_k(\mathrm{H})\to\mathrm{F}_k$
from $p_k(\mathrm{H})$ to a finite group $\mathrm{F}_k$
such that $\varphi_k(p_k(h))\not=\mathrm{id}_{\mathrm{F}_k}$. 
Then, the morphism 
$\varphi_k\circ p_{k\vert\mathrm{H}}\colon\mathrm{H}\to\mathrm{F}_k$
suits.
\end{proof}

Finding the finitely generated subgroups that embed into 
$\mathrm{G}_0^n$ is an important problem, 
in particular related to the theory of foliations 
(representations of holonomy, and cycles in leaves). 
For instance, in \cite{CCGS} one can find:

\begin{thm}[\cite{CCGS}]
{\sl The fundamental group of a compact surface $\Sigma_g$
of genus $g$ embeds into $\mathrm{G}_0^1$, and so 
into $\mathrm{G}_0^n$.

In particular, there are surfaces $S$ with
a foliation $\mathcal{F}$ by curves having an invariant 
curve $\Sigma_g$ with faithfull holonomy representation.}
\end{thm}

Conversely we get the following statement:

\begin{cor}
{\sl The Baumslag-Solitar group $\mathrm{BS}(2,3)$ does not 
embed into $\mathrm{G}_0^n$. 

More generally, if $\mathrm{H}$ is a finitely generated and
non residually finite group, then $\mathrm{H}$ does not
embed into $\mathrm{G}_0^n$.}
\end{cor}

Theorem \ref{thm:fgG0n} has direct applications in 
the theory of holomorphic foliations. 
Toledo constructs smooth complex projective
varieties with fundamental groups which are not 
resi\-dually finite, answering to some Serre's question 
(\cite{Toledo}).
Assume that $\mathcal{F}$ is, for instance, 
a codimension one holomorphic foliation on 
the complex manifold $M$ having an invariant 
variety $N\subset M$ satisfying Toledo's property,
that is $\pi_1(N,\ast)$ not residually finite.
Then the holonomy representation (\cite{CanoCerveauDeserti}) 
\[
\mathrm{Hol}\colon \pi_1(N,\ast)\to\mathrm{Diff}(\mathbb{C},0)
\]
is not faithfull. As we have seen previously, there exist families
of cycles in the leaves of $\mathcal{F}$ near 
the invariant manifold $N$.

In \cite{DrutuSapir} Drutu and Sapir
construct residually finite groups
that are not linear. One of their examples is the 
group $\mathcal{G}=\langle a,\, b\,\vert\, b^2ab^{-2}=a^2\rangle$; 
we prove that $\mathcal{G}$ can not be 
embeded into $\mathrm{G}_0^1$
(or $\widehat{\mathrm{G}}_0^1$):

\begin{pro}
{\sl There is no faithfull representation of 
$\mathcal{G}=\langle a,\,b\,\vert\,b^2ab^{-2}=a^2\rangle$
into $\widehat{\mathrm{G}}_0^1$ $($resp. $\mathrm{G}_0^1)$.}
\end{pro}

\begin{proof}[{\sl Proof}]
Assume by contradiction that there exists a faithfull 
representation $\varphi$ of $\mathcal{G}$ into
$\widehat{\mathrm{G}}_0^1$. Set $A=\varphi(a)$,
and $B=\varphi(b)$. From $B^2AB^{-2}=A^2$ we get that
$A$ and $A^2$ are conjugate (by~$B^2$). In particular, $A$ is 
tangent to the identity, that is $A$ belongs to 
$\widehat{\mathrm{G}}_1^1$; so there exists a formal 
vector field $X$ of order at least $2$ such that~$A=\exp X$
(Proposition \ref{pro:tec1}). Since $\exp X$ and 
$\exp 2X$ are conjugate, the vector fields $X$
and $2X$ are conjugate (by $B^2$). One can assume, 
up to conjugacy, that 
\[
X=\frac{x^{\nu+1}}{1-\lambda x}\frac{\partial}{\partial x}
\]
with $\nu\geq 1$ and $\lambda\in\mathbb{C}$. 
Let $\mu x$ be the linear part of~$B^2$; note that $\mu x$ has to conjugate the 
first non-zero jet $x^{\nu+1}\frac{\partial}{\partial x}$
of $X$ to the first non-zero jet 
$2x^{\nu+1}\frac{\partial}{\partial x}$
of~$2X$. Hence $\mu^{\nu}=2$, and the linear part 
of $B$ is a $2\nu$-th root of $2$; we thus can linearise 
$B$, {\it i.e.} assume that $B=\alpha x$ where 
$\alpha=2^{\frac{1}{2\nu}}$. Let 
$Y=h(x)\frac{\partial}{\partial x}$ be a vector field 
of order $\nu+1$ such that $B_*^2Y=2Y$, that is 
such that $\alpha^{-2}h\big(\alpha^2x\big)=2h$; 
 in other words we have the equality
\begin{equation}\label{eq:bla}
h(\alpha^2x)=2\alpha^2h. 
\end{equation}
Write $h$
as $h=\displaystyle\sum_{\ell\geq\nu+1}h_\ell x^{\ell}$; then 
$(\ref{eq:bla})$ yields to $\alpha^{2\ell}h_\ell=2\alpha^2h_\ell$, {\it i.e.}
$\alpha^{2(\ell-1)}h_\ell=2h_\ell$. For any $h_\ell\not=0$
we get $\alpha^{2(\ell-1)}=2$; in other words 
$2^{\frac{\ell-1}{\nu}-1}=1$, and so 
$2^{\frac{\ell-(\nu+1)}{\nu}}=1;$ as a consequence, 
$\ell=\nu+1$. As a result, in the linearising coordinate for 
$B$, we have: $B=\alpha x$ and 
$A=\exp cx^{\nu+1}\frac{\partial}{\partial x}$
for some $c$. 
In particular the group generated by $A$ and 
$B$ is linear whereas $\mathcal{G}$ is not
(\cite{DrutuSapir}): contradiction.
\end{proof}

In \cite{CerveauLoray} the authors prove the following 
curious result. Let $\gamma$ be an irreducible curve 
in~$\mathbb{P}^2_\mathbb{C}$ of degree~$p^s$ with $p$
prime number. If 
$\varphi\colon\pi_1(\mathbb{P}^2_\mathbb{C}\smallsetminus\gamma,\ast)\to\mathrm{G}_0^1$ or $\varphi\colon\pi_1(\mathbb{P}^2_\mathbb{C}\smallsetminus\gamma,\ast)\to\widehat{\mathrm{G}}_0^1$ 
is a morphism, then the image of $\varphi$ is a finite
group (conjugate to a group of linear rotations, \emph{see} Theorem \ref{thm:clas}); 
moreover, there are some~$\gamma$ such that 
$\pi_1(\mathbb{P}^2_\mathbb{C}\smallsetminus\gamma,\ast)$
contains a free group of rank $2$. That result is used by 
the authors to construct holomorphic first integral for 
codimension one holomorphic foliations in 
$(\mathbb{C}^n,0)$, $n\geq~3$ in special situations 
generalising "Malgrange-Mattei-Moussu Frobenius theorems
with singularities" (\cite{MatteiMoussu, Malgrange}).

\begin{prob}
Let $\gamma$ be a curve in $\mathbb{P}^2_\mathbb{C}$; is 
the group $\pi_1(\mathbb{P}^2_\mathbb{C}\smallsetminus\gamma,\ast)$
a linear group ? is the group
$\pi_1(\mathbb{P}^2_\mathbb{C}\smallsetminus\gamma,\ast)$ 
a residually finite group ?
\end{prob}

\bigskip

\section{Hopfian and co-hopfian groups}\label{sec:hopcohop}

\begin{defi}
A group $\mathrm{G}$ is \textbf{Hopfian} if 
every surjective morphism group from $\mathrm{G}$ 
to~$\mathrm{G}$ is an isomorphism. 

Equivalently, a group is Hopfian if and only if it is not
isomorphic to any of its proper quotients.
\end{defi}

\begin{defi}
A group $\mathrm{G}$ is \textbf{co-Hopfian} if 
every injective morphism group from $\mathrm{G}$ to 
$\mathrm{G}$ is an isomorphism. 

Equivalently, a group is co-Hopfian if and only if it is not
isomorphic to any of its proper subgroups.
\end{defi}

Every finite group is a Hopfian group. Every simple group is 
a Hopfian group. The group $\mathbb{Z}$ of 
integers and the group~$\mathbb{Q}$ of rationals are 
Hopfian groups. However, $\mathbb{C}^*$ is not a Hopfian
group (the morphisms $\mathbb{C}^*\to\mathbb{C}^*$, 
$x\mapsto x^p$ are not injective), and $\mathbb{R}^*$ is 
not a Hopfian group (the morphisms 
$\mathbb{R}^*\to\mathbb{R}^*$, 
$x\mapsto x^p$, $p$ even, are not injective). In \cite{Deserti} 
the author shows that the group 
$\mathrm{Bir}(\mathbb{P}^2_\mathbb{C})$ 
of birational self-maps of the complex
projective plane $\mathbb{P}^2_\mathbb{C}$
is Hopfian.

Let us mention an other statement due to Malcev:

\begin{thm}[\cite{Malcev1, Malcev2}]
{\sl Any finitely generated residually finite group is a 
Hopfian group.}
\end{thm}

\begin{cor}
{\sl Any finitely generated subgroup of $\mathrm{G}_0^n$
$($resp. $\widehat{\mathrm{G}}_0^n)$ is a Hopfian group.}
\end{cor}

Let us now establish the following statement: in which 
the assumption "finitely generated" has been
removed ?

\begin{thm}\label{thm:G00hopfian}
{\sl The group $\widehat{\mathrm{G}}_0^1$ is a Hopfian 
group.}
\end{thm}

To prove it we will use the following result of finite
determination, statement specific to the $1$-dimensional 
and formal case:

\begin{lem}\label{lem:tecconj}
{\sl Let $h$ be an element of $\widehat{\mathrm{G}}_1^1$. 
There exists an integer $\ell$ such that if $g$
belongs to~$\widehat{\mathrm{G}}_\ell^1$, then $h$
and~$hg$ are conjugate in the group 
$\widehat{\mathrm{G}}_1^1$. 

In other words, if two elements of $\widehat{\mathrm{G}}_1^1$
coincide up to a sufficiently large order, then they are 
conjugate.}
\end{lem}

Lemma \ref{lem:tecconj} is a direct consequence of Proposition
\ref{pro:tec1} and Proposition \ref{pro:tec2}.

%

\begin{proof}[{\sl Proof of Theorem \ref{thm:G00hopfian}}]
Let $\varphi\colon\widehat{\mathrm{G}}_0^1\to\widehat{\mathrm{G}}_0^1$
be a surjective morphism. Assume that $\varphi$ is not injective. 
Let $f\in\widehat{\mathrm{G}}_0^1\smallsetminus\{\mathrm{id}\}$
such that $\varphi(f)=\mathrm{id}$. Replacing $f$ by a non-trivial
commutator $[a,f]$ (that also belongs to $\ker\varphi$)
if needed we can assume that $f$ belongs to 
$\widehat{\mathrm{G}}_1^1\smallsetminus\{\mathrm{id}\}$. 
Consider $h$ in $\widehat{\mathrm{G}}_\ell^1$ for $\ell$
sufficiently large. According to Lemma \ref{lem:tecconj} the 
elements $f$ and $hf$ are conjugate, {\it i.e.} there
exists $g\in\widehat{\mathrm{G}}_1^1$ such that $gfg^{-1}=hf$. 
As a consequence, $h=[g,f]$ belongs to $\ker\varphi$. 
Hence, $\ker\varphi$ contains $\widehat{\mathrm{G}}_\ell^1$
for $\ell$ sufficiently large. Since $\varphi$ is surjective,
$\widehat{\mathrm{G}}_0^1$ and
$\newfaktor{\widehat{\mathrm{G}}_0^1}{\ker\varphi}$ are
isomorphic. As $\widehat{\mathrm{G}}_1^\ell$ is contained in 
$\ker\varphi$
the morphism 
$\newfaktor{\widehat{\mathrm{G}}_0^1}{\widehat{\mathrm{G}}_\ell^1}\to\newfaktor{\widehat{\mathrm{G}}_0^1}{\ker\varphi}$ 
is surjective. The group 
$\newfaktor{\widehat{\mathrm{G}}_0^1}{\widehat{\mathrm{G}}_\ell^1}$
is solvable, so does
$\newfaktor{\widehat{\mathrm{G}}_0^1}{\ker\varphi}\simeq\widehat{\mathrm{G}_0^1}$: 
contradiction with the fact that $\widehat{\mathrm{G}_0^1}$
contains free subgroups (\cite{BerthierCerveauLinsNeto}).
The surjective morphism $\varphi$ is thus injective, and 
so an isomorphism.
\end{proof}

\begin{probs}\label{pbs:1et2}
\begin{itemize}
\item[1)] Is the group $\mathrm{G}_0^1$
a Hopfian group ? One way to answer to this question 
is to show that if $\tau\colon\mathrm{G}_0^1\to\mathrm{G}_0^1$
is surjective, then $\tau$ can be extended to a morphism
$\widetilde{\tau}\colon\widehat{\mathrm{G}}_0^1\to\widehat{\mathrm{G}}_0^1$
still surjective.

\item[2)] Are the groups $\mathrm{G}_0^n$ and $\widehat{\mathrm{G}}_0^n$
Hopfian groups ?
\end{itemize}
\end{probs}

Unfortunately, the method used for the proof of 
Theorem \ref{thm:G00hopfian} turns out to be 
ineffective for Pro\-blems~\ref{pbs:1et2}.

Let us now deal with the notion of co-Hopfian group.

Using transcendence basis it is easy to construct
an injective and non-surjective morphism 
$\tau\colon\mathbb{C}\to\mathbb{C}$ of the field 
$\mathbb{C}$; then $\tau$ induces an 
injective and non-surjective homomorphism from 
$\widehat{\mathrm{G}}_0^n$ into itself defined by 
\[
\sum A_Ix^I\mapsto \sum\tau(A_I)x^I
\]
where $A_I$ belongs to $\mathbb{C}^n$. In particular, 
\textsl{$\widehat{\mathrm{G}}_0^n$ is not co-Hopfian.}

\begin{thm}
{\sl The groups $\mathrm{G}_1^n$ and $\widehat{\mathrm{G}}_1^n$
are not co-Hopfian groups.}
\end{thm}

\begin{proof}[{\sl Proof}]
Let us first assume that $n=1$. The morphism 
$\tau_1\colon f\mapsto\tau_1(f)$ defined by 
$\tau_1(f)(x)=\Big(f(x^2)\Big)^{1/2}$ is 
injective but not surjective; indeed, any $\tau_1(f)$
commutes with the involution $x\mapsto -x$ (we choose 
the determination $\sqrt{1}=1$), and for instance
$x+x^2$ does not commute with the involution 
$x\mapsto -x$.

\smallskip

Suppose now that $n>1$. We will use a similar idea 
considering the application
\[
E\colon(x_1,x_2,\ldots,x_n)\mapsto(x_1^2,x_1x_2,x_1x_3,\ldots,x_1x_n)
\]
whose inverse is 
\[
E^{-1}\colon(x_1,x_2,\ldots,x_n)\mapsto\left(\sqrt{x_1},\frac{x_2}{\sqrt{x_1}},\frac{x_3}{\sqrt{x_1}},\ldots,\frac{x_n}{\sqrt{x_1}}\right).
\]
Let us choose the determination of $E^{-1}$ associated
to the principal determination of $\sqrt{\,\,\,}$; the 
application~$\tau_n$ defined by
\[
\tau_n(f)=\tau_n(f_1,f_2,\ldots,f_n)(x)=E^{-1}(f\circ E)(x)
\]
is an injective morphism that is not surjective; indeed
the $\tau_n(f)$ commute with the involution $x\mapsto -x$.
\end{proof}

\begin{prob}
Is the group $\mathrm{G}_0^n$ a co-Hopfian group ?
\end{prob}

\bigskip

\section{Tits alternative}

A group $\mathrm{H}$ satisfies \textbf{Tits alternative}
if for every finitely generated subgroup $\mathrm{K}$
of $\mathrm{H}$: 
\begin{itemize}
\item[$\diamond$] either $\mathrm{K}$ is virtually solvable
 ({\it i.e.} $\mathrm{K}$ contains a solvable subgroup 
 of finite index), 
 
\item[$\diamond$] or $\mathrm{K}$ contains a non-abelian
 free subgroup.
\end{itemize}
 
 Tits proved in \cite{Tits} that linear groups satisfy
 Tits alternative.

\begin{prob}
Do the groups $\mathrm{G}_0^n$ and 
$\widehat{\mathrm{G}}_0^n$ satisfy Tits alternative ?
\end{prob}
 
This important question is related to the Galois 
theory of holomorphic foliations (\cite{Casale1, Casale2}). Note that, 
as it can be seen in \cite{BerthierCerveauLinsNeto},
the group $\mathrm{G}_0^n$ contains free subgroups
of rank $\geq 2$. Furthermore, the solvable non 
abelian subgroups of $\mathrm{G}_0^1$ and 
$\widehat{\mathrm{G}}_0^1$ are classified in
\cite{CerveauMattei}. Solvable subgroups of 
$\mathrm{G}_0^n$, $n>~1$, have been studied 
(\cite{MarteloRibon, Ribon}); in particular,
Rib\'{o}n determines the "length of resolubility" of 
solvable subgroups of $\mathrm{G}_0^n$.

\bigskip

\section{Automorphism groups}\label{sec:autG01}

In \cite{Whittaker} Whittaker proves the following statement:
let $X$ and $Y$ be compact manifolds, with or without boundary, 
and $\varphi$ be a group isomorphism between the group
$\mathrm{Homeo}(X)$ of all homeomorphisms of~$X$ 
into itself and $\mathrm{Homeo}(Y)$, then there exists an homeomorphism $\psi$ 
of~$X$ onto $Y$ such that $\varphi(f)=\psi f\psi^{-1}$
for all $f\in\mathrm{Homeo}(X)$. When $X=Y$ we get that 
every automorphism of $\mathrm{Homeo}(X)$ is an inner 
one. In \cite{Filipkiewicz} Filipkiewicz gives 
a similar result in the context of differentiable manifolds:
let $M$ and $N$ be smooth manifolds without boundary, 
and let $\mathrm{Diff}^p(M)$ denote the group of 
$\mathcal{C}^p$-diffeomorphisms of $M$. The author 
proves that if $\mathrm{Diff}^p(M)$ and $\mathrm{Diff}^q(N)$
are isomorphic as abstract groups, then $p=q$, and the 
isomorphism is induced by a $\mathcal{C}^p$-diffeomorphism
from $M$ to $N$. Let us mention that there are similar 
results in different contexts: \emph{see for instance}
\cite{Banyaga1, Banyaga2, Deserti2}... In particular,
in \cite{Deserti2} the author proves that any 
automorphism of the Cremona group 
$\mathrm{Bir}(\mathbb{P}^2_\mathbb{C})$ of birational
self-maps of the complex projective plane is the 
composition of an inner automorphism and an automorphism
of the field of complex numbers. 
The sketch of the proof is the 
following. Let 
$\mathrm{G}_1$, $\mathrm{G}_2$, $\ldots$, $\mathrm{G}_\ell$
be maximal abelian uncountable subgroups of 
$\mathrm{Bir}(\mathbb{P}^2_\mathbb{C})$, and let $\varphi$ be an 
automorphism of $\mathrm{Bir}(\mathbb{P}^2_\mathbb{C})$;
the author proves that, up to inner conjugacy and the action
of an automorphism of the field $\mathbb{C}$, we 
have $\varphi_{\vert\mathrm{G}_k}=\mathrm{id}$ for 
$1\leq k\leq \ell$, and deduce from it that 
$\varphi_{\vert\mathrm{Bir}(\mathbb{P}^2_\mathbb{C})}=\mathrm{id}$.
A similar strategy will be used in this section to
describe the automorphism groups of $\widehat{\mathrm{G}}_0^1$ and
of $\mathrm{G}_0^1$.

\subsection{Automorphism groups of $\widehat{\mathrm{G}}_0^1$}\label{subsec:autform}

In this section we establish the description of the automorphism 
groups of $\widehat{\mathrm{G}}_0^1$:

\begin{thm}\label{thm:out}
{\sl The group $\mathrm{Aut}(\widehat{\mathrm{G}}_0^1)$ is 
generated by the inner automorphisms and the automorphisms of 
the field $\mathbb{C}$. In other words 
\[
\mathrm{Out}(\widehat{\mathrm{G}}_0^1)\simeq\mathrm{Aut}(\mathbb{C},+,\cdot)
\]
where $\mathrm{Out}(\widehat{\mathrm{G}}_0^1)$ denotes
the non-inner automorphisms of $\widehat{\mathrm{G}}_0^1$.}
\end{thm}

The rest of this section is devoted to the proof of the Theorem \ref{thm:out}.
The study of the maximal abelian subgroups of $\widehat{\mathrm{G}}_0^n$ 
(resp. $\mathrm{G}_0^n$) is essential for the understanding of 
automorphism groups of $\widehat{\mathrm{G}}_0^n$ 
(resp. $\mathrm{G}_0^n$). Unfortunately we are able to study them only 
in the case $\widehat{\mathrm{G}}_0^1$.
Once again the essential argument is the following one: if 
$f$ belongs to 
$\widehat{\mathrm{G}}_k^1\smallsetminus\widehat{\mathrm{G}}_{k+1}^1$,
then~$f$ is conjugate to 
$\exp X_{k,\lambda}=\exp \frac{x^{k+1}}{1+\lambda x^k}\frac{\partial}{\partial x}$
for a certain $\lambda$.
A computation (\cite{CerveauMoussu, CanoCerveauDeserti}) shows that 
the centralizer 
\[
\mathrm{Cent}(\exp X_{k,\lambda},\widehat{\mathrm{G}}_0^1)=\big\{f\in\widehat{\mathrm{G}}_0^1\,\vert\,f\exp X_{k,\lambda}=\exp X_{k,\lambda}f\big\}
\]
of $\exp X_{k,\lambda}$ in $\widehat{\mathrm{G}}_0^1$ coincides with the group
\[
\mathrm{A}_{k,\lambda}=\big\{\exp tX_{k,\lambda}\,\vert\,t\in\mathbb{C}\big\}\times\big\{x\mapsto\xi x\,\vert\,\xi^k=1\big\}.
\]
This group, which is abelian, and so maximal abelian, contains
exactly $(k-1)$ non-trivial torsion elements. 

Let $\kappa$ be in $\mathbb{C}^*$; denote by 
$\underline{\kappa}\colon x\mapsto\kappa x$ the homothety of 
ratio $\kappa$. If $\kappa$ is not a root of unity, then 
\[
\mathrm{Cent}(\underline{\kappa},\widehat{\mathrm{G}}_0^1)=\mathrm{A}_0=\big\{\underline{\mu}\,\vert\,\mu\in\mathbb{C}^*\big\}
\]
which is also a maximal abelian subgroup. A subgroup of $\widehat{\mathrm{G}}_0^1$ 
whose  all elements are pe\-riodic is abelian
and conjugate to a subgroup of $\mathrm{A}_0$ (\emph{see} \cite[Corollary 7.21]{CanoCerveauDeserti}); in parti\-cular,
a maximal abelian subgroup of 
$\widehat{\mathrm{G}}_0^1$ 
contains a non-periodic element. As a consequence,
we get:

\begin{thm}\label{thm:abmaxssgp}
{\sl The maximal abelian subgroups of $\widehat{\mathrm{G}}_0^1$ are
the conjugate of the groups $\mathrm{A}_0$ and 
$\mathrm{A}_{k,\lambda}$ where $k\geq 1$ is an integer,
and $\lambda$ an element of $\mathbb{C}$.}
\end{thm}

Let us now consider 
$\sigma\colon\widehat{\mathrm{G}}_0^1\to\widehat{\mathrm{G}}_0^1$
an automorphism of $\widehat{\mathrm{G}}_0^1$. For instance
the inner automorphisms of~$\widehat{\mathrm{G}}_0^1$
\[
f\mapsto \tau(f)=\varphi f\varphi^{-1}
\]
are such examples. Let $\mathrm{Aut}(\mathbb{C},+,\cdot)$
be the automorphism group of the field $\mathbb{C}$.
Any automorphism $\tau$ of the field $\mathbb{C}$
induces an automorphism~$\sigma^\tau$ that sends 
$f=\displaystyle\sum_{\ell\geq 1}a_\ell x^{\ell}$ to 
$\sigma^\tau(f)=\displaystyle\sum_{\ell\geq 1}\tau(a_\ell)x^{\ell}$.

Note that the image of a maximal abelian subgroup of
$\widehat{\mathrm{G}}_0^1$ by $\sigma$ is still a maximal 
abelian subgroup of~$\widehat{\mathrm{G}}_0^1$.

\begin{thm}
{\sl Let $\sigma$ be an automorphism of $\widehat{\mathrm{G}}_0^1$.
Then, up to a suitable conjugacy, $\sigma(\mathrm{A}_0)=~\mathrm{A}_0$
and~$\sigma(\mathrm{A}_{k,0})=\mathrm{A}_{k,0}$ for any integer 
$k\geq 1$.

Furthermore, if $\lambda$ is non-zero, then 
$\sigma(\mathrm{A}_{k,\lambda})$ is conjugate to 
$\mathrm{A}_{k,\mu}$ for some $\mu$ in $\mathbb{C}^*$.}
\end{thm}

\begin{proof}[{\sl Proof}]
The group $\mathrm{A}_0$ has an infinite number of torsion elements
whereas the $\mathrm{A}_{k,\lambda}$'s don't; this gives the 
first assertion. We can thus assume that $\sigma(\mathrm{A}_0)=\mathrm{A}_0$.
Let us note that $\mathrm{A}_0$ acts by conjugacy on the groups~$\mathrm{A}_{k,0}$: if $\underline{\mu}$ belongs to $\mathrm{A}_0$, 
then 
\[
\underline{\mu}\exp tX_{k,0}\underline{\mu}^{-1}=\exp \mu^{-k}tX_{k,0}.
\]
However, the conjugate of $\mathrm{A}_{k,\lambda}$ by 
$\underline{\mu}$ is $\mathrm{A}_{k,\mu^{-k}\lambda}$, 
and $\mathrm{A}_0$ does not act on the $\mathrm{A}_{k,\lambda}$.
Consequently, $\mathrm{A}_0=\sigma(\mathrm{A}_0)$ acts by 
conjugacy on $\sigma(\mathrm{A}_{k,0})$. Counting the 
torsion elements we get that $\sigma(\mathrm{A}_{k,0})$
is conjugate to $\mathrm{A}_{k,\lambda}$ for some $\lambda$
in $\mathbb{C}$. In particular 
\[
\sigma\Big(\exp x^{k+1}\frac{\partial}{\partial x}\Big)=\exp X
\]
where $X$ is a formal vector field conjugate to  
$X_{k,\lambda}$ for some $\lambda$ in $\mathbb{C}$.

\begin{lem}\label{lem:tec1}
{\sl One has: $X=a_kx^{k+1}\frac{\partial}{\partial x}$ for some 
non-zero complex number $a_k$.}
\end{lem}

\begin{proof}[{\sl Proof of Lemma \ref{lem:tec1}}]

Let $\mu$ be a complex number that is not a root of unity. Then, 
$\mu^\prime=\mu^{\sigma}$ is not a root of unity. We have
\[
\sigma\left(\mu\Big(\exp x^{k+1}\frac{\partial}{\partial x}\Big)\mu^{-1}\right)=\mu^{\prime}(\exp X)(\mu^{\prime})^{-1}=\exp\mu^\prime_*X
\]
and 
\[
\sigma\left(\mu\Big(\exp x^{k+1}\frac{\partial}{\partial x}\Big)\mu^{-1}\right)=\sigma\left(\exp \mu_* x^{k+1}\frac{\partial}{\partial x}\right)=\sigma\left(\exp \mu^{-k} x^{k+1}\frac{\partial}{\partial x}\right)
\]
From 
\[
\sigma\{\exp tx^{k+1}\frac{\partial}{\partial x}\,\vert\,t\in\mathbb{C}\}=\{\exp sX\,\vert\,s\in\mathbb{C}\}
\]
one gets 
\[
\sigma\left(\exp \mu^{-k}x^{k+1}\frac{\partial}{\partial x}\right)=\exp sX
\]
for some $s$ (dependent on $\mu$) wich finally implies $sX=\mu^\prime_*X$.
We write $X$ as $X=\displaystyle\sum_{\ell\geq k}a_\ell x^\ell\frac{\partial}{\partial x}$, 
$a_k\not=0$, then 
\[
\mu^\prime_*X=\displaystyle\sum_{\ell\geq k}a_\ell\mu^{1-\ell}x^\ell\frac{\partial}{\partial x}.
\]
The relation $sX=\mu^\prime_*X$ implies 
$sa_\ell=(\mu^\prime)^{1-\ell}a_\ell$ for $l\geq k$.
Since $a_k\not=0$, if $a_\ell\not=0$ for an $\ell>k$, then
$(\mu^\prime)^{1-\ell}=s=(\mu^\prime)^{1-k}$, and $\mu^\prime$ is a root of 
unity: contradiction.
\end{proof}

As a result, $\sigma(\mathrm{A}_{k,0})=\mathrm{A}_{k,0}$ for
any $k$; the torsion elements can optionally be swapped 
and $\exp tx^{k+1}\frac{\partial}{\partial x}$ is sent 
onto $\exp t_{\sigma_k}x^{k+1}\frac{\partial}{\partial x}$.
Let us remark that 
$\mathbb{C}\ni t\mapsto t_{\sigma_k}\in\mathbb{C}$ is
an additive morphism group.
\end{proof}

Hence an automorphism $\sigma$ of $\widehat{\mathrm{G}}_0^1$
induces a multiplicative isomorphism of $\mathbb{C}^*$
\[
\sigma(\underline{\lambda})=\lambda^\sigma,
\]
and an additive isomorphism $\sigma_k$ of $\mathbb{C}$ 
for any $k$ 
\[
\sigma_k(t)=t_{\sigma_k}.
\]
Let us come back to the action of $\mathrm{A}_0$ on 
$\mathrm{A}_{k,0}$
\[
\Big(\underline{\lambda},\exp tx^{k+1}\frac{\partial}{\partial x}\Big)\mapsto
\exp \lambda^ktx^{k+1}\frac{\partial}{\partial x}
\]
that corresponds to the action of $\mathbb{C}^*$
on $\mathbb{C}$
\[
(\lambda,t)\mapsto \lambda^kt.
\]
The action is transformed by the automorphism $\sigma$ 
into 
\[
\Big(\lambda^\sigma,\exp t_{\sigma_k}x^{k+1}\frac{\partial}{\partial x}\Big)\mapsto\exp (\lambda^\sigma)^kt_{\sigma_k}x^{k+1}\frac{\partial}{\partial x}
\]
but also into 
\[
\Big(\lambda^\sigma,\exp t_{\sigma_k}x^{k+1}\frac{\partial}{\partial x}\Big)\mapsto\exp (\lambda^kt)_{\sigma_k}x^{k+1}\frac{\partial}{\partial x}.
\]
Therefore 
\[
(\lambda^\sigma)^kt_{\sigma_k}=(\lambda^kt)_{\sigma_k};
\]
in particular, for $t=1$ we get
\[
s^\sigma 1_{\sigma_k}=s_{\sigma_k}.
\]
Hence $s\mapsto s^\sigma$ is an automorphism of the
field $\mathbb{C}$, and the additive morphisms $s\mapsto 
s_{\sigma_k}$ differ from $s^\sigma$ only by a 
multiplicative constant $1_{\sigma_k}$. Up to the 
action of the automorphism of $\widehat{\mathrm{G}}_0^1$
associated to this field automorphism we can assume that
$s\mapsto s^\sigma$ is the identity, and that 
$s_{\sigma_k}=\varepsilon_ks$ where~$\varepsilon_k$
denotes a non-zero constant.

\begin{lem}\label{lem:G01aut}
{\sl If $\sigma$ belongs to $\mathrm{Aut}(\widehat{\mathrm{G}}_0^1)$, then
$\sigma(\widehat{\mathrm{G}}_k^1)=\widehat{\mathrm{G}}_k^1$.}
\end{lem}

\begin{proof}[{\sl Proof}]
Let $h$ be an element of $\widehat{\mathrm{G}}_k^1$; then there
exist $\varphi$ in $\widehat{\mathrm{G}}_0^1$, $p\geq k$, and
$\lambda$ in $\mathbb{C}$ such that 
\[
h=\varphi\exp X_{p,\lambda}\varphi^{-1}.
\]
Recall that $\sigma(\exp X_{p,\lambda})=\exp aX_{p,\lambda'}$ for
some $a\in\mathbb{C}^*$ and $\lambda'\in\mathbb{C}$. 
Hence 
\[
\sigma(h)=\sigma(\varphi)\,\exp aX_{p,\lambda'}\,\sigma(\varphi^{-1})
\]
belongs to $\widehat{\mathrm{G}}_k^1$; similarly $\sigma^{-1}(h)$
belongs to $\widehat{\mathrm{G}}_k^1$.
\end{proof}

\begin{rem}
From Lemma \ref{lem:G01aut} we get that $\sigma$ is a continuous
automorphism of $\widehat{\mathrm{G}}_0^1$ endowed with the
Krull topology.
\end{rem}

Let us recall the Baker-Campbell-Hausdorff formula applied to 
the formal vector fields $X=a(x)\frac{\partial}{\partial x}$
and $Y=b(x)\frac{\partial}{\partial x}$ of 
$\widehat{\chi}_2^1$ (\emph{see for instance}
\cite{Hall}): If $Z\in\widehat{\chi}_2^1$
is a solution of 
\[
\exp Z=\exp X\exp Y,
\]
then 
\[
Z=X+Y+\frac{1}{2}[X,Y]+\frac{1}{12}[X,[X,Y]]-\frac{1}{12}[Y,[X,Y]]+\ldots
\]
In particular 
\begin{equation}\label{eq:expcom}
\big(\exp X\big)\big(\exp Y\big)\big(\exp -X\big)\big(\exp -Y\big)=\exp\big([X,Y]+\text{h.o.t.}\big)
\end{equation}
where h.o.t. denotes terms of order $\geq 3$ in the 
algebra generated by $X$ and $Y$. We thus get
\[
\left(\exp x^{k+1}\frac{\partial}{\partial x}\right)\left(\exp x^{\ell+1}\frac{\partial}{\partial x}\right)\left(\exp -x^{k+1}\frac{\partial}{\partial x}\right)\left(\exp -x^{\ell+1}\frac{\partial}{\partial x}\right)=\exp\Big(\Big[\exp x^{k+1}\frac{\partial}{\partial x},\exp x^{\ell+1}\frac{\partial}{\partial x}\Big]+\chi\Big)
\]
where $\chi$ is a vector field of the form
\[
\chi=x^{\ell+k+1+\inf(k,\ell)}a(x)\frac{\partial}{\partial x};
\]
in other words
\begin{eqnarray*}
\left(\exp x^{k+1}\frac{\partial}{\partial x}\right)\left(\exp x^{\ell+1}\frac{\partial}{\partial x}\right)\left(\exp -x^{k+1}\frac{\partial}{\partial x}\right)\left(\exp -x^{\ell+1}\frac{\partial}{\partial x}\right)&=&\exp\Big((\ell-k)x^{k+\ell+1}\frac{\partial}{\partial x}+\chi\Big)\\
&=&\exp\Big((\ell-k)x^{k+\ell+1}\frac{\partial}{\partial x}\Big)h
\end{eqnarray*}
where $h$ denotes an element of
$\widehat{\mathrm{G}}^1_{\ell+k+\inf(k,\ell)}$ (still by
Baker-Campbell-Hausdorff formula). Applying $\sigma$ we get 
($\varepsilon_k$ has been introduced just before
Lemma \ref{lem:G01aut})
\begin{eqnarray*}
& & \left(\exp \varepsilon_kx^{k+1}\frac{\partial}{\partial x}\right)\left(\exp \varepsilon_\ell x^{\ell+1}\frac{\partial}{\partial x}\right)\left(\exp -\varepsilon_kx^{k+1}\frac{\partial}{\partial x}\right)\left(\exp -\varepsilon_\ell x^{\ell+1}\frac{\partial}{\partial x}\right)\\
& &\hspace{1cm}=\exp\Big(\varepsilon_k\varepsilon_\ell\Big[x^{k+1}\frac{\partial}{\partial x},x^{\ell+1}\frac{\partial}{\partial x}\Big]+\widetilde{\chi}\Big)\\
& &\hspace{1cm}=\exp\Big(\varepsilon_{k+\ell}x^{k+\ell+1}\frac{\partial}{\partial x}\Big)\sigma(h)
\end{eqnarray*}
where $\widetilde{\chi}$ is given by (\ref{eq:expcom}) and 
$\sigma(h)$ is controlled by Lemma \ref{lem:G01aut}.
By truncating to a suitable order we see that 
$\varepsilon_k\varepsilon_\ell=\varepsilon_{k+\ell}$. Up to 
the action by an homothety we can assume that
$\varepsilon_1=1$; as a result $\varepsilon_\ell=\varepsilon_{\ell+1}$.
Hence by induction $\varepsilon_k=1$ for any $k$.
The automorphism $\sigma$ thus fixes the homotheties and the 
$\exp tx^{k+1}\frac{\partial}{\partial x}$. The 
quotient groups 
$\newfaktor{\widehat{\mathrm{G}}_0^1}{\widehat{\mathrm{G}}_k^1}$
are generated by the projections of the homotheties and by the
projections of the $\exp tx^{k+1}\frac{\partial}{\partial x}$, 
$t\in\mathbb{C}$, $\ell\leq k-1$. According to Lemma 
\ref{lem:G01aut} the automorphism $\sigma$ induces an 
automorphism of the quotient groups 
$\newfaktor{\widehat{\mathrm{G}}_0^1}{\widehat{\mathrm{G}}_k^1}$
that coincides with the identity on any 
$\newfaktor{\widehat{\mathrm{G}}_0^1}{\widehat{\mathrm{G}}_k^1}$.
Therefore, $\sigma$ coincides with the identity.
This ends the proof of Theorem \ref{thm:out}.

\subsection{Automorphisms groups of $\mathrm{G}_0^1$}

The aim of this section is to give the following 
description of the automorphisms groups of $\mathrm{G}_0^1$:

\begin{thm}\label{thm:autGO1}
{\sl The group $\mathrm{Out}(\mathrm{G}_0^1)$ is isomorphic
to $\newfaktor{\mathbb{Z}}{2\mathbb{Z}}$.}
\end{thm}

We are going to adapt the above approach to the holomorphic
case, {\it i.e.} to the description of 
$\mathrm{Out}(\mathrm{G}_0^1)$. Theorem \ref{thm:abmaxssgp}
has no analogue but \cite{Ecalle} and \cite{CerveauMoussu} imply:

\begin{thm}\label{thm:maxabsub}
{\sl Let $\mathrm{A}$ be an uncountable maximal abelian subgroup 
$\mathrm{A}$ of $\mathrm{G}_0^1$ such that 
$\mathrm{A}\cap\mathrm{G}_1^1$ is uncoun\-table. Then, 
up to conjugacy,
\begin{itemize}
    \item[$\diamond$] either $\mathrm{A}=\mathrm{A}_0$,

    \item[$\diamond$] or $\mathrm{A}=\mathrm{A}_{k,\lambda}$ where $k\geq 1$
denotes an integer, and $\lambda$ a complex number.
\end{itemize}

Moreover, the group $\mathrm{A}_0$ is a maximal abelian
subgroup of $\mathrm{G}_0^1$.}
\end{thm}

\begin{rems}
\begin{enumerate}
\item The group  
$\mathrm{G}_1^1$ contains many countable 
abelian maximal subgroups;
in fact the group generated by a
"generic" element of $\mathrm{G}_1^1$ is maximal
(\cite{Ecalle}). 

\item There are diffeomorphisms 
$f=\lambda z+\text{h.o.t}$, $\lambda=\exp(2\mathbf{i}\pi\gamma)$,
$\gamma\in\mathbb{R}\smallsetminus\mathbb{Q}$, 
which are not holomorphically linearisable. 
The maximal abelian group that contains $f$
is not necessarily conjugate to 
$\mathrm{A}_0$ or $\mathrm{A}_{k,\lambda}$. It can be 
uncountable (see \cite{PerezMarco}).
\end{enumerate}
\end{rems}

Let $\sigma$ be an element of $\mathrm{Aut}(\mathrm{G}_0^1)$. 
We can, as in the formal case, characterize the 
groups~$\mathrm{A}_0$ and $\mathrm{A}_{k,\lambda}$
by their torsion elements. Recall that if $f$ belongs
to $\mathrm{G}_1^1$, then $f$ is a commutator, and
$\sigma(f)$ also. As a consequence, 
$\sigma(\mathrm{G}_1^1)=\mathrm{G}_1^1$, and if
$\mathrm{H}$ is a maximal abelian subgroup 
of~$\mathrm{G}_1^1$, then $\sigma(\mathrm{H})$ 
also. Hence, as in the formal context, 
$\sigma(\mathrm{A}_{k,0})$ is holomorphically 
conjugate to one of the $\mathrm{A}_{k,\lambda}$.
We want to prove that~$\sigma(\mathrm{A}_0)$ is 
conjugated to $\mathrm{A}_0$.

\begin{lem}\label{lem:per}
{\sl Let $f$ be an element of $\mathrm{G}_0^1$. Assume 
that $\mathrm{A}_{k,\lambda}$ is invariant by conjugation
by $f$, 
{\it i.e.} 
$f\mathrm{A}_{k,\lambda}f^{-1}=\mathrm{A}_{k,\lambda}$.
Then 
\begin{itemize}
\item[$\diamond$] if $\lambda\not=0$, then 
$f=\xi\exp\varepsilon\frac{x^{p+1}}{1+\lambda x^p}\frac{\partial}{\partial x}$ for some $\varepsilon$, $\xi$ in $\mathbb{C}$ such that $\xi^p=1$;

\item[$\diamond$] if $\lambda=0$, then 
$f=\mu\exp\varepsilon x^{p+1}\frac{\partial}{\partial x}$ 
for some $\mu$ in $\mathbb{C}^*$, and 
$\varepsilon$ in $\mathbb{C}$.
\end{itemize}}
\end{lem}

\begin{proof}[{\sl Proof}]
Left to the reader.
\end{proof}

Now, let us fixed $k$; by Theorem \ref{thm:maxabsub} we 
can assume, up to 
conjugacy, that $\sigma(\mathrm{A}_{k,0})=\mathrm{A}_{k,\lambda}$ for some $\lambda$. But $\sigma(\mathrm{A}_0)$
contains periodic elements of all periods so 
$\lambda=0$, and $\sigma(\mathrm{A}_0)$ 
is an abelian subgroup of the "affine" group 
$\big\{\mu\exp\varepsilon x^{p+1}\frac{\partial}{\partial x}\,\vert\,\mu\in\mathbb{C}^*,\,\varepsilon\in\mathbb{C}\big\}$ (Lemma \ref{lem:per}). Such a group is 
either conjugate to a subgroup of~$\mathrm{A}_0$, or conjugate to a subgroup of 
\[
\Big\{\xi\exp\varepsilon\frac{x^{p+1}}{1+\lambda x^p}\frac{\partial}{\partial x}\,\vert\,\varepsilon,\,\xi\in\mathbb{C},\,\xi^p=1\Big\}. 
\]
The fact that 
$\sigma(\mathrm{A}_0)$ contains an infinite 
number of periodic elements implies that 
$\sigma(\mathrm{A}_0)$ is a subgroup of~$\mathrm{A}_0$. As $\sigma(\mathrm{A}_0)$
is maximal, one gets: $\sigma(\mathrm{A}_0)=\mathrm{A}_0$. 

Recall that $\sigma(\mathrm{A}_0)=\mathrm{A}_0$ 
acts on all the $\sigma(\mathrm{A}_{\ell,0})$; 
an argument similar to that used in the formal case 
(\S\ref{subsec:autform})
implies that 
$\sigma(\mathrm{A}_{\ell,0})=\mathrm{A}_{\ell,0}$ 
for all $\ell\in\mathbb{N}$.

As before $\sigma\colon\mathrm{A}_0\to\mathrm{A}_0$ 
corresponds to an automorphism of the field 
$\mathbb{C}$.

\begin{lem}
{\sl Any $\sigma\in\mathrm{Aut}(\mathrm{G}_0^1)$ extends 
into an automorphism 
$\widehat{\sigma}\colon\widehat{\mathrm{G}}_0^1\to\widehat{\mathrm{G}}_0^1$. 
Moreover, $\sigma(\mathrm{G}_k^1)=\mathrm{G}_k^1$.}
\end{lem}

\begin{proof}[{\sl Proof}]
The projections of the homotheties and the 
$\exp x^\ell\frac{\partial}{\partial x}$'s, $\ell\leq k-1$,
generate 
$\newfaktor{\mathrm{G}_0^1}{\mathrm{G}_0^k}\simeq\newfaktor{\widehat{\mathrm{G}}_0^1}{\widehat{\mathrm{G}}_0^k}$.
Since~$\sigma$ preserves $\mathrm{A}_0$ and the 
$\mathrm{A}_{\ell,0}$'s, it  
induces an automorphism 
\[
\sigma_k\colon\newfaktor{\widehat{\mathrm{G}}_0^1}{\widehat{\mathrm{G}}_0^k}\to\newfaktor{\widehat{\mathrm{G}}_0^1}{\widehat{\mathrm{G}}_0^k}
\]
for any $k$. By construction these automorphisms
are compatible with the filtration induced by the 
$\mathrm{G}_0^k$'s
 \[
 \xymatrix{
    \newfaktor{\widehat{\mathrm{G}}_0^1}{\widehat{\mathrm{G}}_0^k} \ar[rr]^{\sigma_k} & & \newfaktor{\widehat{\mathrm{G}}_0^1}{\widehat{\mathrm{G}}_0^k} \\
    \newfaktor{\widehat{\mathrm{G}}_0^1}{\widehat{\mathrm{G}}_0^{k+\ell}}\ar[u] \ar[rr]_{\sigma_{k+\ell}} & & \newfaktor{\widehat{\mathrm{G}}_0^1}{\widehat{\mathrm{G}}_0^{k+\ell}}\ar[u]
  }
\]
As a result, the $\sigma_k$'s determine an automorphism
$\widehat{\sigma}\colon\widehat{\mathrm{G}}_0^1\to\widehat{\mathrm{G}}_0^1$.
This automorphism extends~$\sigma$ in the following sense: if
we fix a coordinate $z$, we get an embedding 
$\mathrm{G}_0^1\hookrightarrow\widehat{\mathrm{G}}_0^1$
that associates to the convergent element $f=\sum a_nz^n$ 
the element $f=\sum a_nz^n$ seen as a formal series. 
By construction $\widehat{\sigma}(f)=\sigma(f)$, {\it i.e.}
$\widehat{\sigma}$ sends a holomorphic diffeomorphism onto 
a holomorphic diffeomorphism.
\end{proof}

According to Theorem \ref{thm:out} we can assume that 
$\widehat{\sigma}$ is associated to an automorphism
$\tau$ of the field $\mathbb{C}$: if $f=\sum a_nz^n$, then 
\[
\widehat{\sigma}(f)=\displaystyle\sum_{n\geq 1} \tau(a_n)z^n.
\]
The automorphism $\tau$ satisfies the following property:
if $\displaystyle\sum_{n\geq 1}a_nz^n$ converges, then 
$\displaystyle\sum_{n\geq 1}\tau(a_n)z^n$ also converges.

\begin{lem}
{\sl Either $\tau$ is the identity $z\mapsto z$, or $\tau$ is 
the complex conjugation $z\mapsto\overline{z}$. }
\end{lem}

\begin{proof}[{\sl Proof}]
If $\tau$ preserves $\mathbb{R}$, that is if 
$\tau(\mathbb{R})=\mathbb{R}$, then $\tau$ is 
either the identity, or the complex conjugation.

If $\tau(\mathbb{R})\not=\mathbb{R}$, then the 
image of the unit disk $\mathbb{D}(0,1)$ 
is dense in $\mathbb{C}$ (\emph{see} \cite{Kestelman}).
Suppose that $\tau$ is neither the identity, 
nor the complex conjugation. By density there 
exists $a_n\in\mathbb{D}(0,1)$, $n\geq 2$, 
such that $\vert\tau(a_n)-n!\vert<1$. If 
$f=z+\sum a_nz^n$, then on the one hand $f$ belongs to 
$\mathrm{G}_0^1$, and on the other hand $z+\sum a_nz^n$ diverges:
contradiction.
\end{proof}

\bigskip

\section{Automorphisms groups of $\widehat{\mathrm{G}}_0^n$}\label{sec:autG0n}

We are not able to prove an analogue of Theorem
\ref{thm:out} in higher dimension, nevertheless we obtain the following 
partial result:

\begin{pro}\label{pro:res}
{\sl Up to conjugacy and up to the action of an element of 
$\mathrm{Aut}(\mathbb{C},+,\cdot)$, the restriction 
of an element of 
$\mathrm{Aut}(\widehat{\mathrm{G}}_0^n)$ to $\mathrm{GL}(\mathbb{C}^n)$ is the identity map.}
\end{pro}

Denote by 
\begin{align*}
& \mathrm{j}^1\colon\widehat{\mathrm{G}}_0^n\to\mathrm{GL}(\mathbb{C}^n), && f\mapsto Df_{(0)}
\end{align*}
the map that associates to $f$ its linear part at $0$. Let $\sigma$ be an element of 
$\mathrm{Aut}(\widehat{\mathrm{G}}_0^n)$, and 
set $\widetilde{\varphi}=\mathrm{j}^1\circ\sigma$, {\it i.e.}
\begin{eqnarray*}
\widetilde{\varphi}\,\,\,\colon\,\,\,\widehat{\mathrm{G}}_0^n &\stackrel{\sigma}{\longrightarrow}& \widehat{\mathrm{G}}_0^n\quad \stackrel{\mathrm{j}^1}{\longrightarrow}\quad \mathrm{GL}(\mathbb{C}^n)\\ 
f &\mapsto& \sigma(f) \,\,\,\,\,\mapsto\quad D(\sigma(f))_{(0)}
\end{eqnarray*}
and consider $\varphi=\widetilde{\varphi}_{\vert\mathrm{GL}(\mathbb{C}^n)}$ the 
restriction of $\widetilde{\varphi}$ to $\mathrm{GL}(\mathbb{C}^n)$:
\begin{align*}
&\varphi\colon\mathrm{GL}(\mathbb{C}^n)\to\mathrm{GL}(\mathbb{C}^n), && A\mapsto D(\sigma(A))_{(0)}.
\end{align*}

Denote by $\mathcal{H}$ the subgroup of homotheties 
$\mathbb{C}^*\mathrm{id}$ of $\mathrm{GL}(\mathbb{C}^n)$.

Proposition \ref{pro:res} follows from the following facts:
\begin{itemize}
\item[$\diamond$] The normal subgroup $\ker\varphi$ of $\mathrm{GL}(\mathbb{C}^n)$ is a 
subgroup of $\mathcal{H}$.

\item[$\diamond$] The map $\varphi$ is an injective morphism.

\item[$\diamond$] Consider a non-periodic homothety $\lambda\,\mathrm{id}$.
The image $\varphi(\lambda\,\mathrm{id})$ of $\lambda\,\mathrm{id}$ by 
$\varphi$ is an homothety.

\item[$\diamond$] If $\sigma$ belongs to 
$\mathrm{Aut}(\widehat{\mathrm{G}}_0^n)$, 
then up to conjugacy $\sigma_{\vert\mathrm{GL}(\mathbb{C}^n)}$ is an automorphism 
of~$\mathrm{GL}(\mathbb{C}^n)$.
\end{itemize}

\begin{rem}
Proposition \ref{pro:res} can be used as follows. Take an element 
$A$ in $\mathrm{GL}(n,\mathbb{C})$, typically 
$A_0=(\lambda x_1,x_2,x_3,\ldots,x_n)$, and consider the group
$\mathrm{Cent}(A,\widehat{\mathrm{G}}_0^n)=\big\{f\in\widehat{\mathrm{G}}_0^n\,\vert\, f\circ A=A\circ f\big\}$. In the case of the example of $A_0$, for some 
generic $\lambda$, we have
\[
\mathrm{Cent}(A_0,\widehat{\mathrm{G}}_0^n)=\big\{\big(a(x_2,x_3,\ldots,x_n)x_1,g(x_2,x_3,\ldots,x_n)\big)\,\vert\, a\in\mathbb{C}[[x_2,x_3,\ldots,x_n]]^*,\,g\in\widehat{\mathrm{G}}_0^{n-1}\big\}.
\]
If $\sigma\in\mathrm{Aut}(\widehat{\mathrm{G}}_0^n)$ is such that 
$\sigma_{\vert\mathrm{GL}(n,\mathbb{C})}=\mathrm{id}_{\mathrm{GL}(n,\mathbb{C})}$,
then 
$\sigma\big(\mathrm{Cent}(A,\widehat{\mathrm{G}}_0^n)\big)=\mathrm{Cent}(A,\widehat{\mathrm{G}}_0^n)$,
and we can expect for instance with $A$ of type $A_0$ to
use an induction on the dimension to prove that the automorphisms group of $\widehat{\mathrm{G}}_0^n$ is 
generated by inner automorphisms, and the automorphisms of the 
field $\mathbb{C}$.
\end{rem}

\section{The $\mathcal{C}^\infty$-case}

Let $\mathrm{Diff}^\infty(\mathbb{R}^n,0)$ be the group
of germs of $\mathcal{C}^\infty$-diffeomorphisms of 
$(\mathbb{R}^n,0)$, and let 
$\widehat{\mathrm{Diff}(\mathbb{R}^n,0)}$
be its completion. Consider the map 
$\mathcal{T}\colon \mathrm{Diff}^\infty(\mathbb{R}^n,0)\to\widehat{\mathrm{Diff}(\mathbb{R}^n,0)}$
that sends the diffeomorphism $f$ onto the 
infinite Taylor expansion 
$\mathcal{T}(f)=\displaystyle\sum_{n=0}^\infty\frac{f^{(n)}(0)}{n!}x^n$. Denote by $\mathrm{Diff}^{\infty}_{\infty}(\mathbb{R}^n,0)$ 
the kernel of $\mathcal{T}$ consisting of diffeomorphisms 
infinitely tangent to the identity.

\begin{pro}
Let $\mathrm{G}$ be a finitely generated subgroup of
$\mathrm{Diff}^\infty(\mathbb{R}^n,0)$. If the 
restriction of $\mathcal{T}$ to $\mathrm{G}$ is 
one-to-one, then $\mathrm{G}$ is residually finite.
\end{pro}

\begin{proof}
The group $\mathrm{G}$ is isomorphic to a 
finitely generated subgroup of 
$\widehat{\mathrm{Diff}(\mathbb{R}^n,0)}$; 
hence $\mathrm{G}$ is residually finite.
\end{proof}

\begin{question}
Let $\mathrm{G}$ be a finitely generated subgroup
of $\mathrm{Diff}_\infty^\infty(\mathbb{R}^n,0)$;
is $\mathrm{G}$ residually finite ?
\end{question}

Let us mention a result attributed to Thurston and 
"reproved" by Reeb and Schweitzer (\cite{ReebSchweitzer}) which 
could be useful to answer the question in dimension $1$:

\begin{thm}[Thurston]
Let $\mathrm{G}$ be a finitely generated subgroup
of $\mathrm{Diff}_\infty^\infty(\mathbb{R}^n,0)$.

There exists a non-trivial morphism from $\mathrm{G}$
to $\mathbb{R}$.
\end{thm}

\vspace{2cm}

\bibliographystyle{alpha}
\bibliography{biblio}

\begin{thebibliography}{BCLN96}

\bibitem[AI88]{ArnoldIlyashenko}
V.~I. Arnol'd and Yu.~S. Il'yashenko.
\newblock Ordinary differential equations [{\it {c}urrent problems in
  mathematics. {f}undamental directions, {v}ol. 1}, 7--149, {A}kad. {N}auk
  {SSSR}, {V}sesoyuz. {I}nst. {N}auchn. i {T}ekhn. {I}nform., {M}oscow, 1985;
  {MR}0823489 (87e:34049)].
\newblock In {\em Dynamical systems, {I}}, volume~1 of {\em Encyclopaedia Math.
  Sci.}, pages 1--148. Springer, Berlin, 1988.
\newblock Translated from the Russian by E. R. Dawson and D. O'Shea.

\bibitem[Arn88]{Arnold}
V.~I. Arnol'd.
\newblock {\em Geometrical methods in the theory of ordinary differential
  equations}, volume 250 of {\em Grundlehren der mathematischen Wissenschaften
  [Fundamental Principles of Mathematical Sciences]}.
\newblock Springer-Verlag, New York, second edition, 1988.
\newblock Translated from the Russian by Joseph Sz\"{u}cs [J\'{o}zsef M.
  Sz\H{u}cs].

\bibitem[Ban86]{Banyaga1}
A.~Banyaga.
\newblock On isomorphic classical diffeomorphism groups. {I}.
\newblock {\em Proc. Amer. Math. Soc.}, 98(1):113--118, 1986.

\bibitem[Ban97]{Banyaga2}
A.~Banyaga.
\newblock {\em The structure of classical diffeomorphism groups}, volume 400 of
  {\em Mathematics and its Applications}.
\newblock Kluwer Academic Publishers Group, Dordrecht, 1997.

\bibitem[BCLN96]{BerthierCerveauLinsNeto}
M.~Berthier, D.~Cerveau, and A.~Lins~Neto.
\newblock Sur les feuilletages analytiques r\'{e}els et le probl\`eme du
  centre.
\newblock {\em J. Differential Equations}, 131(2):244--266, 1996.

\bibitem[BS62]{BaumslagSolitar}
G.~Baumslag and D.~Solitar.
\newblock Some two-generator one-relator non-{H}opfian groups.
\newblock {\em Bull. Amer. Math. Soc.}, 68:199--201, 1962.

\bibitem[Cas06]{Casale1}
G.~Casale.
\newblock Feuilletages singuliers de codimension un, groupo\"{\i}de de {G}alois
  et int\'{e}grales premi\`eres.
\newblock {\em Ann. Inst. Fourier (Grenoble)}, 56(3):735--779, 2006.

\bibitem[Cas11]{Casale2}
G.~Casale.
\newblock An introduction to {M}algrange pseudogroup.
\newblock In {\em Arithmetic and {G}alois theories of differential equations},
  volume~23 of {\em S\'{e}min. Congr.}, pages 89--113. Soc. Math. France,
  Paris, 2011.

\bibitem[CCD13]{CanoCerveauDeserti}
F.~Cano, D.~Cerveau, and J.~D\'{e}serti.
\newblock {\em Th\'eorie \'el\'ementaire des feuilletages holomorphes
  singuliers}.
\newblock Echelle. Belin, 2013.

\bibitem[CCGS20]{CCGS}
S.~Cantat, D.~Cerveau, V.~Guirardel, and J.~Souto.
\newblock Surface groups in the group of germs of analytic diffeomorphisms in
  one variable.
\newblock {\em Enseign. Math.}, 66(1-2):93--134, 2020.

\bibitem[CL98]{CerveauLoray}
D.~Cerveau and F.~Loray.
\newblock Un th\'{e}or\`eme de {F}robenius singulier via l'arithm\'{e}tique
  \'{e}l\'{e}mentaire.
\newblock {\em J. Number Theory}, 68(2):217--228, 1998.

\bibitem[CLPT19]{ClaudonLorayPereiraTouzet}
B.~Claudon, F.~Loray, J.~V. Pereira, and F.~Touzet.
\newblock Holonomy representation of quasi-projective leaves of codimension one
  foliations.
\newblock {\em Publ. Mat.}, 63(1):295--305, 2019.

\bibitem[CM82]{CerveauMattei}
D.~Cerveau and J.-F. Mattei.
\newblock {\em Formes int\'{e}grables holomorphes singuli\`eres}, volume~97 of
  {\em Ast\'{e}risque}.
\newblock Soci\'{e}t\'{e} Math\'{e}matique de France, Paris, 1982.
\newblock With an English summary.

\bibitem[CM88]{CerveauMoussu}
D.~Cerveau and R.~Moussu.
\newblock Groupes d'automorphismes de {$({\bf C},0)$} et \'{e}quations
  diff\'{e}rentielles {$ydy+\cdots=0$}.
\newblock {\em Bull. Soc. Math. France}, 116(4):459--488 (1989), 1988.

\bibitem[Coh95]{Cohen}
S.~D. Cohen.
\newblock The group of translations and positive rational powers is free.
\newblock {\em Quart. J. Math. Oxford Ser. (2)}, 46(181):21--93, 1995.

\bibitem[D\'06]{Deserti2}
J.~D\'{e}serti.
\newblock Sur les automorphismes du groupe de {C}remona.
\newblock {\em Compos. Math.}, 142(6):1459--1478, 2006.

\bibitem[D\'07]{Deserti}
J.~D\'{e}serti.
\newblock Le groupe de {C}remona est hopfien.
\newblock {\em C. R. Math. Acad. Sci. Paris}, 344(3):153--156, 2007.

\bibitem[Die55]{Dieudonne}
J.~Dieudonn\'{e}.
\newblock {\em La g\'{e}om\'{e}trie des groupes classiques}.
\newblock Ergebnisse der Mathematik und ihrer Grenzgebiete, (N.F.), Heft 5.
  Springer-Verlag, Berlin-G\"{o}ttingen-Heidelberg, 1955.

\bibitem[DS05]{DrutuSapir}
C.~Dru\c{t}u and M.~Sapir.
\newblock Non-linear residually finite groups.
\newblock {\em J. Algebra}, 284(1):174--178, 2005.

\bibitem[\'{E}75]{Ecalle}
J.~\'{E}calle.
\newblock Th\'{e}orie it\'{e}rative: introduction \`a la th\'{e}orie des
  invariants holomorphes.
\newblock {\em J. Math. Pures Appl. (9)}, 54:183--258, 1975.

\bibitem[Fil82]{Filipkiewicz}
R.~P. Filipkiewicz.
\newblock Isomorphisms between diffeomorphism groups.
\newblock {\em Ergodic Theory Dynam. Systems}, 2(2):159--171 (1983), 1982.

\bibitem[Hal15]{Hall}
B.~Hall.
\newblock {\em Lie groups, {L}ie algebras, and representations}, volume 222 of
  {\em Graduate Texts in Mathematics}.
\newblock Springer, Cham, second edition, 2015.
\newblock An elementary introduction.

\bibitem[Kes51]{Kestelman}
H.~Kestelman.
\newblock Automorphisms of the field of complex numbers.
\newblock {\em Proc. London Math. Soc. (2)}, 53:1--12, 1951.

\bibitem[Mal40]{Malcev1}
A.~Malcev.
\newblock On isomorphic matrix representations of infinite groups.
\newblock {\em Rec. Math. N.S.}, 8 (50):405--422, 1940.

\bibitem[Mal65]{Malcev2}
A.~Malcev.
\newblock {On the faithful representation of infinite groups by matrices}.
\newblock {\em {Transl., Ser. 2, Am. Math. Soc.}}, 45:1--18, 1965.

\bibitem[Mal76]{Malgrange}
B.~Malgrange.
\newblock Frobenius avec singularit\'{e}s. {I}. {C}odimension un.
\newblock {\em Inst. Hautes \'{E}tudes Sci. Publ. Math.}, (46):163--173, 1976.

\bibitem[Mal82]{Malgrangebourbaki}
B.~Malgrange.
\newblock Travaux d'\'{E}calle et de {M}artinet-{R}amis sur les syst\`emes
  dynamiques.
\newblock In {\em Bourbaki {S}eminar, {V}ol. 1981/1982}, volume~92 of {\em
  Ast\'{e}risque}, pages 59--73. Soc. Math. France, Paris, 1982.

\bibitem[MM80]{MatteiMoussu}
J.-F. Mattei and R.~Moussu.
\newblock Holonomie et int\'{e}grales premi\`eres.
\newblock {\em Ann. Sci. \'{E}cole Norm. Sup. (4)}, 13(4):469--523, 1980.

\bibitem[MR14]{MarteloRibon}
M.~Martelo and J.~Rib\'{o}n.
\newblock Derived length of solvable groups of local diffeomorphisms.
\newblock {\em Math. Ann.}, 358(3-4):701--728, 2014.

\bibitem[Per22]{Pereira}
J.~V. Pereira.
\newblock Closed meromorphic 1-forms.
\newblock {\em ar{X}iv:2206.09745}, 2022.

\bibitem[PM95]{PerezMarco}
R.~P\'{e}rez~Marco.
\newblock Nonlinearizable holomorphic dynamics having an uncountable number of
  symmetries.
\newblock {\em Invent. Math.}, 119(1):67--127, 1995.

\bibitem[Rib19]{Ribon}
J.~Rib\'{o}n.
\newblock The solvable length of groups of local diffeomorphisms.
\newblock {\em J. Reine Angew. Math.}, 752:105--139, 2019.

\bibitem[RS78]{ReebSchweitzer}
G.~Reeb and P.~Schweitzer.
\newblock Un th\'{e}or\`eme de {T}hurston \'{e}tabli au moyen de l'analyse non
  standard.
\newblock In {\em Differential topology, foliations and {G}elfand-{F}uks
  cohomology ({P}roc. {S}ympos., {P}ontif\'{\i}cia {U}niv. {C}at\'{o}lica,
  {R}io de {J}aneiro, 1976)}, volume Vol. 652 of {\em Lecture Notes in Math.},
  pages p. 138. Lecture Notes in Mathematics, Vol. 652. Springer, Berlin-New
  York, 1978.

\bibitem[Sie42]{Siegel1}
C.~L. Siegel.
\newblock Iteration of analytic functions.
\newblock {\em Ann. of Math. (2)}, 43:607--612, 1942.

\bibitem[Sie52]{Siegel2}
C.~L. Siegel.
\newblock \"{U}ber die {N}ormalform analytischer {D}ifferentialgleichungen in
  der {N}\"{a}he einer {G}leichgewichtsl\"{o}sung.
\newblock {\em Nachr. Akad. Wiss. G\"{o}ttingen. Math.-Phys. Kl.
  Math.-Phys.-Chem. Abt.}, 1952:21--30, 1952.

\bibitem[Tit72]{Tits}
J.~Tits.
\newblock Free subgroups in linear groups.
\newblock {\em J. Algebra}, 20:250--270, 1972.

\bibitem[Tol93]{Toledo}
D.~Toledo.
\newblock Projective varieties with non-residually finite fundamental group.
\newblock {\em Inst. Hautes \'{E}tudes Sci. Publ. Math.}, (77):103--119, 1993.

\bibitem[Whi63]{Whittaker}
J.~V. Whittaker.
\newblock On isomorphic groups and homeomorphic spaces.
\newblock {\em Ann. of Math. (2)}, 78:74--91, 1963.

\bibitem[Yoc95]{Yoccoz}
J.-C. Yoccoz.
\newblock Th\'{e}or\`eme de {S}iegel, nombres de {B}runo et polyn\^{o}mes
  quadratiques.
\newblock Number 231, pages 3--88. 1995.
\newblock Petits diviseurs en dimension $1$.

\end{thebibliography}

\nocite{}

\end{document}